\documentclass[11pt]{amsart}
\usepackage{a4wide,amssymb,showlabels}
\usepackage[all]{xy}

\newcommand{\w}{\omega}
\newcommand{\U}{\mathcal U}
\newcommand{\D}{\mathcal D}
\newcommand{\IR}{\mathbb R}
\newcommand{\IQ}{\mathbb Q}
\newcommand{\IN}{\mathbb N}
\newcommand{\e}{\varepsilon}
\newcommand{\K}{\mathcal K}
\newcommand{\C}{\mathcal C}
\newcommand{\F}{\mathcal F}
\newcommand{\cs}{\mathsf{cs}}
\newcommand{\ap}{\mathsf{ap}}
\newcommand{\as}{\mathsf{as}}
\newcommand{\Ra}{\Rightarrow}
\newcommand{\Cld}{\mathsf{Cld}}
\newcommand{\dom}{\mathrm{dom}}

\newtheorem{theorem}{Theorem}[section]
\newtheorem{example}[theorem]{Example}
\newtheorem{proposition}[theorem]{Proposition}
\newtheorem{problem}[theorem]{Problem}
\newtheorem{corollary}[theorem]{Corollary}
\newtheorem{lemma}[theorem]{Lemma}

\title{The cometrizability of generalized metric spaces}
\author{Taras Banakh and Yaryna Stelmakh}
\address{T.Banakh: Ivan Franko National University of Lviv (Ukraine) and Institute of Mathematics, Jan Kochanowski University in Kielce (Poland)}
\email{t.o.banakh@gmail.com}
\address{Ya. Stelmakh:  Ivan Franko National University of Lviv, Ukraine}
\email{yarynziya@ukr.net}

\begin{document}
\begin{abstract}
A topological space $X$ is cometrizable if it admits a weaker metrizable topology such that each point $x\in X$ has a (not necessarily open) neighborhood base consisting of metrically closed sets. We study the relation of cometrizable spaces to other generalized metric spaces and prove that all  $\mathsf{as}$-cosmic spaces are cometrizable. Also, we present an example of a regular countable space of weight $\omega_1$, which is not cometrizable. Under $\w_1=\mathfrak c$ this space contains no infinite compact subsets and hence is $\cs$-cosmic. Under $\w_1<\mathfrak p$ this countable space is Fr\'echet-Urysohn and is not $\cs$-cosmic.
\end{abstract}
\maketitle

\section{Introduction}
{\parskip1pt

In this paper we study the interplay between the class of cometrizable spaces and other classes of generalized metric spaces. 

A topological space $X$ is called {\em cometrizable} if it admits a metrizable topology such that each point $x\in X$ has a (not necessarily open) neighborhood base consisting of metrically closed sets. Equivalently, cometrizable spaces can be defined as spaces $X$ for which there exists a bijective continuous map $f:X\to M$ to a metrizable space $M$ such that for any open set $U\subseteq X$ and point $x\in U$ there exists a neighborhood $V\subseteq X$ of $x$ such that $f^{-1}(\overline{f(V)})\subseteq U$. 

It is clear that each cometrizable space $X$ is regular and {\em submetrizable}, i.e. admits a continuous bijective map onto a metrizable space.

Cometrizable spaces were introduced by Gruenhage \cite{Grue89} who proved the following  interesting implication of PFA, the Proper Forcing Axiom \cite{Baum}.

\begin{theorem}[Gruenhage] Under PFA, a cometrizable space $X$ is cosmic if and only if $X$ contains no uncountable  discrete subspaces and no uncountable subspaces of the Sorgenfrey line.
\end{theorem} 

In \cite[8.5]{Tod} Todor\v cevi\'c proved that this PFA-characterization of cosmic cometrizable  spaces remains true under OCA (the Open Coloring Axiom, which follows from PFA).

These results of Gruenhage and Todor\v cevi\'c motivate a deeper study of cometrizable spaces. In this paper we establish some inheritance properties of the class of cometrizable spaces and using the obtained information study the relation of cometrizable space to some known classes of generalized metric spaces.

We start with the following simple (but important) observation.

\begin{proposition}\label{p:sub} The class of cometrizable spaces contains all metrizable spaces and is closed under taking subspaces and countable Tychonoff products.
\end{proposition}

Next we show that the class of cometrizable spaces is stable under forming  function spaces $C_\K(X,Y)$. 

Here for topological spaces $X,Y$ by $C(X,Y)$ we denote the set of all continuous functions from $X$ to $Y$. Given a family $\mathcal K$ of compact subsets of the space $X$ by $C_\K(X,Y)$ we denote the space $C(X,Y)$ endowed with the topology generated by the subbase consisting of the sets 
$$[K,U]:=\{f\in C(X,Y):f(K)\subseteq U\}$$where $K\in\K$ and $U$ is an open set in $Y$.

If $\K$ is the family of all compact (finite) subsets of $X$, then the function space $C_\K(X,Y)$ will be denoted by $C_{k}(X,Y)$ (resp. $C_p(X,Y)$). If $\K$ is the family of convergent sequences in $X$ (i.e., countable compact sets with a unique non-isolated point), then the function space $C_\K(X,Y)$ will be denoted by $C_{\cs}(X,Y)$. The function spaces  $C_k(X,\IR)$, $C_\cs(X,\IR)$ and $C_p(X,\IR)$ are denoted by $C_{k}(X)$, $C_\cs(X)$ and $C_p(X)$, respectively. 

A family $\K$ of compact subsets of a topological space $X$ is called {\em separable} if 
\begin{itemize}\itemsep1pt
\item each compact subset of any compact set $K\in\K$ belongs to the family $\K$;
\item the union $\bigcup\K$ is dense in $X$;
\item $\K$ contains a countable subfamily $\C$ such that each compact set $K\in\K$ can be enlarged to a compact set $\tilde K\in\K$ such that $\tilde K\cap\bigcup\C$ is dense in $\tilde K$.
\end{itemize}

\begin{theorem}\label{t:Ck} Let $X$ be a topological space and $\K$ be a separable family of compact subsets of $X$. Then for any cometrizable space $Y$ the function space $C_\K(X,Y)$ is cometrizable.
\end{theorem}

This theorem will be proved in Section~\ref{s:Ck}. Now we derive some its corollaries.
\smallskip

A topological space $X$ is defined to be
\begin{itemize}\itemsep=2pt
\item {\em $\cs$-separable} if $X$ contains a countable set $D\subseteq X$, which is {\em sequentially dense} in $X$ in the sense that each point $x\in X$ is the limit of some convergent sequence $\{x_n\}_{n\in\w}\subseteq D$; 
\item {\em $\mathsf k$-separable} if $X$ contains a countable subset $D\subseteq X$, which is {\em $k$-dense}  in the sense that each compact set $K\subseteq X$ is contained in a compact set $C\subseteq X$ such that $K\subseteq\overline{C\cap D}$;
\item {\em $\sigma \mathsf k$-separable} if $X$ contains a $\sigma$-compact set $D\subseteq X$ such that each compact set $K\subseteq X$ is contained in a compact set $C\subseteq X$ such that $K\subseteq \overline{C\cap D}$.
\end{itemize}

It is clear that each $\mathsf k$-separable space $X$ is $\sigma\mathsf k$-separable (and $\cs$-separable if all compact subsets of $X$ are Fr\'echet-Urysohn). In \cite{GR}, $\sigma\mathsf k$-separable spaces are called spaces with property (CK).

Theorem~\ref{t:Ck} has two corollaries: 

\begin{corollary}\label{c1} For any $\sigma\mathsf k$-separable topological space $X$ and any cometrizable space $Y$ the function space $C_{k}(X,Y)$ is cometrizable.
\end{corollary} 

\begin{corollary}\label{c2} For any $\cs$-separable topological space $X$ and any cometrizable space $Y$ the function space $C_{\cs}(X,Y)$ is cometrizable.
\end{corollary}
  
Corollary~\ref{c1} should be compacred with the following characterization of $\sigma\mathsf k$-separabe spaces, proved by Gartside and Reznichenko in \cite[Theorem 16]{GR}. 

\begin{theorem}[Gartside, Reznichenko] A Tychonoff space $X$ is $\sigma\mathsf k$-separable if and only if its function space $C_k(X)$ is cometrizable.
\end{theorem}

In \cite[Corollary 11]{GR}, Gartside and Reznichenko characterized Tychonoff space with cometrizable spaces $C_p(X)$ as follows.

\begin{theorem}[Gartside, Reznichenko] For a Tychonoff space $X$ the following conditions are equivalent:
\begin{enumerate}
\item the function space $C_p(X)$ is cometrizable;
\item $C_p(X)$ contains a dense cometrizable subspace;
\item $C_p(X)$ is metrizable;
\item $X$ is countable.
\end{enumerate}
\end{theorem}

Now we study the relation of the class of cometrizable spaces to other known classes of generalized metric spaces.  We start with stratifiable spaces and their generalizations.

A regular topological space $X$ is called
\begin{itemize}\itemsep=2pt
\item {\em stratifiable}  if each point $x\in X$ has a countable system of open neighborhoods $(U_n(x))_{n\in\w}$ such that each closed subset $F$ of $X$ is equal to $\bigcap_{n\in\w}\overline{U_n[F]}$ where $U_n[F]=\bigcup_{x\in F}U_n(x)$;
\item {\em semi-stratifiable}   if each point $x\in X$ has a countable system of open neighborhoods $(U_n(x))_{n\in\w}$ such that each closed subset $F$ of $X$ is equal to $\bigcap_{n\in\w}U_n[F]$;
\item {\em quarter-stratifiable} if there exists a function $U:X\times\IN\to\tau$ such that $X=\bigcup_{x\in X}U(x,n)$ for all $n\in\w$ and for any point $x\in X$, any  sequence $\{x_n\}_{n\in\w}\subseteq X$ with $x\in\bigcap_{n\in\w}U(x_n,n)$ converges to $x$;
\item a space with {\em $G_\delta$-diagonal} if the diagonal $\{(x,x):x\in X\}$ is a $G_\delta$-set in $X\times X$.
\end{itemize}
Stratifiable spaces were introduced by Borges \cite{Borges} but were known earlier as $M_3$-spaces of Ceder \cite{Ceder}. Semi-stratifiable spaces were introduced by Greede \cite{Greede} and quarter-stratifiable spaces by Banakh, who proved in \cite{Ban02} that every semi-stratifiable space is quarter-stratifiable and  every quarter-stratifiable space has $G_\delta$-diagonal. More information on stratifiable and semi-stratifiable spaces can be found in \cite{Grue} and \cite{Grue2}. The following important result was proved by Gartside and Reznichenko in \cite[Proposition 1]{GR}.

\begin{theorem}[Gartside, Reznichenko] \label{t:S} Each stratifiable space is cometrizable.
\end{theorem}

Next, we recall some local properties of topological spaces.

A topological space $X$ is defined to be
\begin{itemize}\itemsep=2pt
\item {\em Fr\'echet-Urysohn} if for any subset $A\subseteq X$ and point $x\in\bar A$ there exists a sequence $\{x_n\}_{n\in\w}\subseteq A$ that conveges to $x$;
\item {\em sequential} if for any non-closed set $A\subseteq X$ there exists a sequence $\{x_n\}_{n\in\w}\subseteq A$ that converges to a point $x\in X\setminus A$;
\item a {\em $k$-space} if for any non-closed set $A\subseteq X$ there exists a compact set $K\subseteq X$ such that $K\cap A$ is not closed in $K$;
\item a {\em $k_\IR$-space} if the continuity of a function $f:X\to\IR$ is equivalent to the continuity of its restrictions $f{\restriction}K$ onto compact subsets $K$ on $X$;
\smallskip

\item {\em Ascoli} if for any compact subset $K\subseteq C_{k}(X)$ the map $K\times X\to\IR$, $(f,x)\mapsto f(x)$, is continuous;
\item {\em $\cs$-Ascoli} (or {\em sequentially Ascoli}) if for any convergent sequence $K\subseteq C_{k}(X)$ the map $K\times X\to\IR$, $(f,x)\mapsto f(x)$, is continuous.
\end{itemize} 
By a {\em convergent sequence} we understand a compact countable set with a unique non-isolated point. Ascoli spaces were introduced and studied in \cite{BG}.
By \cite[5.4]{BG} (and \cite[5.8]{BG}), a (Tychonoff) space $X$ is Ascoli if and only if the canonical map $\delta:X\to C_{k}(C_{k}(X))$ assigning to each $x\in X$ the Dirac measure $\delta_x:f\mapsto f(x)$ is continuous (if and only if the map $\delta$ is a topological embedding). Sequentially Ascoli spaces were studied in \cite{Gab1}, \cite{Gab2}  and in \cite{BanFan} (as spaces containing no strict $\mathsf{Cld}^\w$-fans).  For any Tychonoff space $X$ we have the implications:
$$\mbox{Fr\'echet-Urysohn $\Ra$ sequential $\Ra$ $\mathsf k$-space $\Ra$ $\mathsf k_\IR$-space $\Ra$ Ascoli $\Ra$ $\cs$-Ascoli}.$$
The unique non-trivial implication ($k_\IR$-space $\Ra$ Ascoli) in this diagram is due to Noble \cite{Noble}, see \cite[\S5]{BG}. By \cite[2.18]{Gab2}, any non-discrete $P$-space is $\cs$-Ascoli but not Ascoli. 

There is a useful characterization of $\cs$-Ascoli spaces in terms of (strict) $\Cld^\w$-fans, which are defined as follows.

A sequence $(F_n)_{n\in\w}$ of {\em closed} subsets of a topological space $X$ is called 
\begin{itemize}\itemsep=2pt
\item  {\em compact-finite} in $X$ if each compact subset $K\subseteq X$ the set $\{n\in\w:K\cap F_n\ne\emptyset\}$ is finite;
\item {\em strictly compact-finite} in $X$ if each set $F_n$ has a functionally open neighborhood $U_n\subseteq X$ such that for any compact set $K\subseteq X$ the set $\{F\in\F:K\cap U_n\ne\emptyset\}$ is finite;
\item {\em accumulating at} a point $x\in X$ if for each neighborhood $O_x$ of $x$ the set $\{n\in\w:O_x\cap F_n\ne\emptyset\}$ is infinite;
\item a ({\em strict})  {\em $\Cld^\w$-fan} in $X$ if it is (strictly) compact-finite and accumulates at some point $x\in X$.
\end{itemize}
A subset $U$ of a topological space $X$ is called a {\em functionally open neighborhood} of a set $A\subseteq X$ if there exists a continuous function $f:X\to[0,1]$ such that $f(F)\subseteq\{0\}$ and $f(O\setminus A)\subseteq\{1\}$.

The following characterization of sequentially Ascoli spaces was proved in \cite[3.3.1, 2.9.6]{BanFan} (for the equivalence $(1)\Leftrightarrow(2)$, see also \cite[2.1]{GKP} and \cite{Gab2}).

\begin{theorem}\label{t:A} 
For a topological space $X$ the following conditions are equivalent:
\begin{enumerate}\itemsep=2pt
\item is sequentially Ascoli;
\item contains no strict $\Cld^\w$-fans.
\end{enumerate}
If $X$ is a normal $\aleph$-space, then the conditions \textup{(1), (2)} are equivalent to 
\begin{enumerate}
\item[(3)] $X$ contains no $\Cld^\w$-fans.
\end{enumerate}
\end{theorem}

The class of $\aleph$-spaces, appearing in Theorem~\ref{t:A} was introduced by O'Meara  \cite{OM} and is one of many known classes of generalized metric spaces, which are defined with the help of  networks.

\smallskip

A family $\mathcal N$ of subsets of a topological space $X$ is  called
\begin{itemize}\itemsep=2pt
\item a {\em network} if for each open set $U\subseteq X$ and point $x\in U$ there exists a set $N\in\mathcal N$ such that $x\in N\subseteq U$;
\item a {\em $k$-network} if for each open set $U\subseteq X$ and compact set  $K\subseteq U$ there exists a finite subfamily $\mathcal F\subseteq\mathcal N$ such that $K\subseteq\bigcup\mathcal F\subseteq U$;
\item a {\em Pytkeev network} if $\mathcal N$ is a network and for any open set $U\subseteq X$, a subset $A\subseteq X$ and point $x\in U\cap\bar A\setminus A$,  there exists a set $N\in\mathcal N$ such that $N\subseteq U$ and $N\cap A$ is infinite;
\item an {\em $\mathsf{ap}$-network} if for each open set $U\subseteq X$ and a sequence  $\{x_n\}_{n\in\w}\subseteq X$ of points  that accumulate at a point $x\in U$ there exists a set $N\in\mathcal N$ such that $ N\subseteq U$ and the set $\{n\in\w:x_n\in N\}$ is infinite;
\item an {\em $\mathsf{as}$-network} if  for each open set $U\subseteq X$ and a sequence $(S_n)_{n\in\w}$ of closed subsets of $X$ that accumulates at a point $x\in U$ there exists a set $N\in\mathcal N$ such that $N\subseteq U$ and the set $\{n\in \w:N\cap S_n\ne\emptyset\}$ is infinite.
\end{itemize}
The prefixes $\cs$, $\ap$, $\as$ are abbreviations of ``$\mathsf c$onvergent $\mathsf s$equence'', ``$\mathsf a$ccumulating sequence of $\mathsf p$oints'',  and ``$\mathsf a$ccumulating sequence of closed $\mathsf s$ets''.


A regular topological space $X$ is called
\begin{itemize}\itemsep=2pt
\item an {\em $\aleph_0$-space} if $X$ has a countable $k$-network;
\item an {\em $\aleph$-space} if $X$ has a $\sigma$-discrete $k$-network;
\smallskip
\item {\em cosmic} if $X$ has a countable network;
\item {\em $\cs$-cosmic} if $X$ has a countable $\cs$-network;
\item {\em $\mathsf{ap}$-cosmic} if $X$ has a countable $\mathsf{ap}$-network;
\item {\em $\mathsf{as}$-cosmic}  if $X$ has a countable $\mathsf{as}$-network;
\smallskip
\item a {\em $\sigma$-space} if $X$ has a $\sigma$-discrete network;
\item a {\em $\sigma_{\cs}$-space}  if $X$ has a $\sigma$-discrete $\cs$-network;
\item a {\em $\sigma_{\mathsf{ap}}$-space}  if $X$ has a $\sigma$-discrete $\ap$-network;
\item a {\em $\sigma_{\mathsf{as}}$-space} if $X$ has a $\sigma$-discrete $\as$-network.
\end{itemize}

The classes of cosmic spaces and $\aleph_0$-spaces are two well-studied classes of generalized metric spaces, introduced by Michael \cite{Mich} and considered in the surveys of Gruenhage \cite{Grue}, \cite{Grue2}, \cite{Tan}; $\aleph$-spaces  were introduced and studied by O'Meara \cite{OM} in his dissertation (see also \cite{Foged} and \cite{Grue}). It can be shown (see \cite{Foged}) that a $\sigma$-discrete (more generally, compact-countable) family of sets is a $k$-network if and only if it is a $\cs$-network. This implies that a topological space is an $\aleph_0$-space (resp. an $\aleph$-space) if and only if it is $\cs$-cosmic (res. a $\sigma_\cs$-space). 

The class of $\mathsf{ap}$-cosmic spaces coincides with the class of $\mathfrak P_0$-spaces of Banakh in \cite{Ban15} and the class of $\sigma_{\mathsf{ap}}$-spaces (properly) contains the class of $\mathfrak P$-spaces of Gabriyelyan and K\c akol \cite{GK}. 

The classes of $\mathsf{as}$-cosmic spaces and $\sigma_{\mathsf{as}}$-spaces are new and are introduced in this paper with purpose to find in the class of cometrizable spaces a subclass of spaces defined by suitable network properties. 

\begin{theorem}\label{t:New} Each $\cs$-network $\mathcal N$ in a $\cs$-Ascoli normal $\aleph$-space $X$ is an $\ap$-network. Consequently, every $\cs$-Ascoli normal $\aleph$-space $X$ is a $\sigma_{\ap}$-space.
\end{theorem}

\begin{proof} 
To show that the $\cs$-network $\mathcal N$ is an $\as$-network, fix an open set $U\subseteq X$ and a sequence $(F_n)_{n\in \w}$ of closed subsets of $X$ that accumulates at some point $x\in U$. By Theorem~\ref{t:A}, the $\cs$-Ascoli $\aleph$-space $X$ contains no $\Cld^\w$-fans. Consequently, the sequence $(F_n)_{n\in\w}$ is not a $\Cld^\w$-fan and hence it is not compact-finite. So, we can find a compact set $K\subseteq X$ such that the set $\Omega=\{n\in\w:K\cap F_n\ne\emptyset\}$ is infinite. For each $n\in\Omega$ choose a point $x_n\in K\cap F_n$. Since compact sets in $\aleph$-spaces are metrizable (by \cite[2.4, 4.6]{Grue}), the sequence $(x_n)_{n\in\Omega}$ has a convergent subsequence $(x_{n_k})_{k\in\w}$. Since $\mathcal N$ is a $\cs$-network, there exists a set $N\in\mathcal N$ such that $N\subseteq U$ and the set $\{k\in\w:x_{n_k}\in N\}\subseteq \{k\in \w:N\cap F_{n_k}\ne\emptyset\}$ is infinite. Then the set $\{n\in\w:N\cap F_n\ne\emptyset\}$ is infinite too, witnessing that $\mathcal N$ is an $\as$-network for $X$. 
\end{proof} 

Let us also formulate a corollary of Theorem~\ref{t:S} and a result of Foged \cite{Foged} (see \cite[11.4]{Grue}) (saying that each Fr\'echet-Urysohn $\aleph$-space is stratifiable).

\begin{proposition}\label{p:FU} Each Fr\'echet-Urysohn $\aleph$-space is stratifiable and hence cometrizable.
\end{proposition}

The location of cometrizable spaces among other generalized metric spaces is shown in Diagram~\ref{diag} holding for any regular topological space. By simple arrows we denote the implications that hold under some additional assumptions (written at the arrows).
}
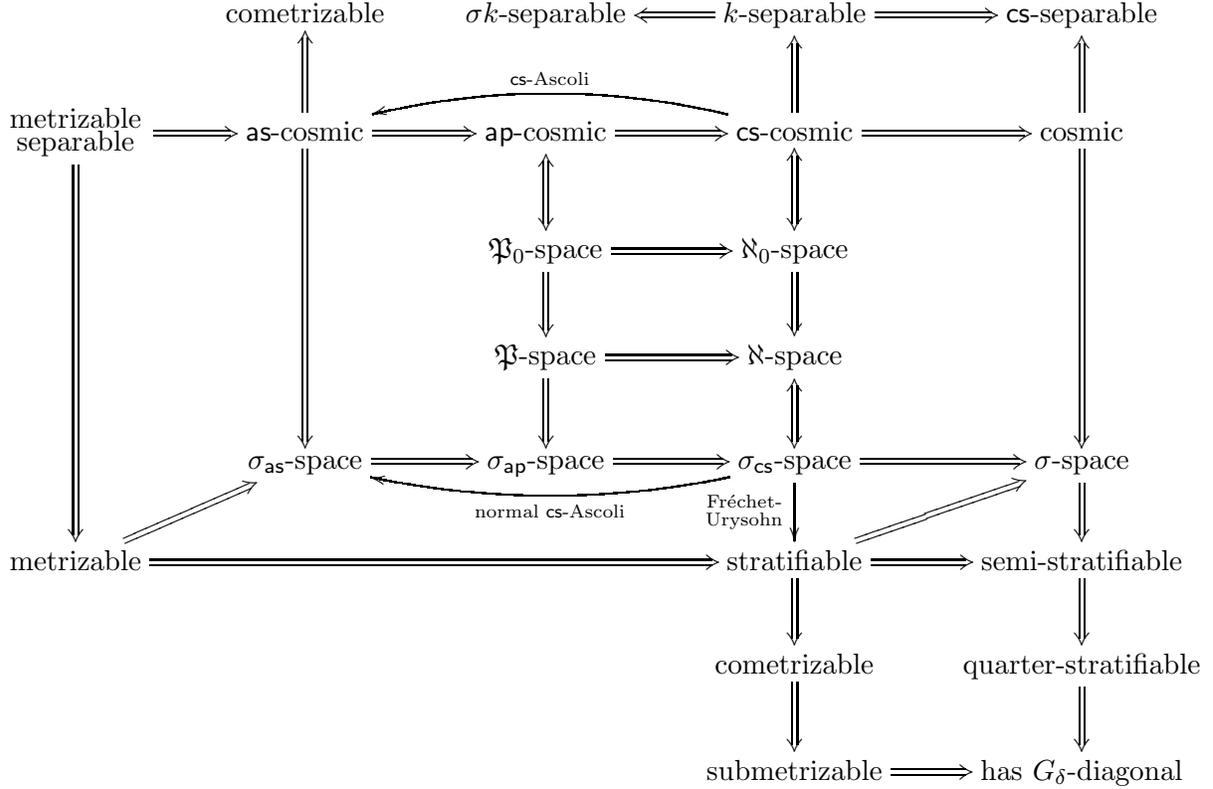
\begin{figure}[h]
$$
\xymatrix{
&\mbox{cometrizable}&\mbox{$\sigma k$-separable}&\mbox{$k$-separable}\ar@{=>}[r]\ar@{=>}[l]&\mbox{$\cs$-separable}\\
\mbox{metrizable}\atop\mbox{separable}\ar@{=>}[r]\ar@{=>}[dddd]
&\mbox{$\as$-cosmic}\ar@{=>}[r]\ar@{=>}[ddd]\ar@{=>}[u]&\mbox{$\ap$-cosmic}\ar@{=>}[r]\ar@{<=>}[d]&\mbox{$\cs$-cosmic}\ar@{=>}[r]\ar@{<=>}[d]\ar@{=>}[u]\ar@/_15pt/_{\mbox{\tiny $\cs$-Ascoli}}[ll]&\mbox{cosmic}\ar@{=>}[ddd]\ar@{=>}[u]\\
&&\mbox{$\mathfrak P_0$-space}\ar@{=>}[d]\ar@{=>}[r]&\mbox{$\aleph_0$-space}\ar@{=>}[d]&\\
&&\mbox{$\mathfrak P$-space}\ar@{=>}[d]\ar@{=>}[r]&\mbox{$\aleph$-space}\ar@{<=>}[d]&\\
&\mbox{$\sigma_{\mathsf{as}}$-space}\ar@{=>}[r]&\mbox{$\sigma_{\mathsf{ap}}$-space}\ar@{=>}[r]&\mbox{$\sigma_{\cs}$-space}\ar@{=>}[r]\ar@/^13pt/^{\mbox{\tiny normal $\cs$-Ascoli}}[ll]\ar_{\tiny \mbox{Fr\'echet-}\atop\mbox{Urysohn}}[d]&\mbox{$\sigma$-space}\ar@{=>}[d]\\
\mbox{metrizable}\ar@{=>}[ru]\ar@{=>}[rrr]&&&\mbox{stratifiable}\ar@{=>}[ur]\ar@{=>}[r]\ar@{=>}[d]&\mbox{semi-stratifiable}\ar@{=>}[d]\\
&&&\mbox{cometrizable}\ar@{=>}[d]&\mbox{quarter-stratifiable}\ar@{=>}[d]\\
&&&\mbox{submetrizable}\ar@{=>}[r]&\mbox{has $G_\delta$-diagonal}
}
$$
\caption{Location of cometrizable spaces among other generalized metric spaces}\label{diag}
\end{figure}
The non-trivial (or not discussed sofar) implications of this diagram are established in the following theorem that will be proved in Section~\ref{s:main}.

\begin{theorem} \label{t:main}
\begin{enumerate}\itemsep=2pt
\item Each cosmic space is $\cs$-separable.
\item Each ${\cs}$-cosmic space is $\mathsf k$-separable.
\item Each $\mathsf{as}$-cosmic space is cometrizable.
\end{enumerate}
\end{theorem}

Now we present some examples showing which implications cannot be added to the above diagram. Let us recall that the {\em Sorgenfrey line} is the real line endowed with the (first-countable) topology, generated by the base consisting of the half-intervals $[a,b)$ where $a<b$ are real numbers.  

\begin{example}\label{ex1} The Sorgenfrey line is cometrizable, first-countable, $\mathsf k$-separable and quarter-stratifiable, but not semi-stratifiable.
\end{example}

The cometrizability of the Sorgenfrey line is witnessed by the standard Euclidean topology of the real line. The $\mathsf k$-separability of the Sorgenfrey line is established in \cite{BGR} (see also Example 20 in \cite{GR}).  By Example 3.2 in \cite{Ban02}, the Sorgenfrey line is quarter-stratifiable but not semi-stratifiable. 

\begin{problem} Is each cometrizable space quarter-stratifiable?
\end{problem}

Our next example is more difficult and will be constructed in Section~\ref{s:ex}.

\begin{example}\label{ex} There exists a regular countable space $X$ of weight $\omega_1$, which is not cometrizable and hence not stratifiable. If $\w_1=\mathfrak c$, then the space $X$ contains no infinite compact sets and is $\cs$-cosmic. If $\w_1<\mathfrak p$, the space $X$ is Fr\'echet-Urysohn and is not $\cs$-cosmic.
\end{example}

The cardinal $\mathfrak p$ is defined as the smallest character of a countable space with a unique non-isolated point, which is not Fr\'echet-Urysohn. It is known that $\w_1\le\mathfrak p\le\mathfrak c$ and $\mathfrak p=\mathfrak c$ under Martin's Axiom, see \cite{vD}, \cite{Vau}, \cite{Blass}. 



Looking at Proposition~\ref{p:FU} and Theorems~\ref{t:New} and \ref{t:main}(3), it is natural to ask

\begin{problem} Is each (sequential) $\sigma_{\as}$-space cometrizable?
\end{problem}

\section{Proof of Theorem~\ref{t:Ck}}\label{s:Ck}

Given a separable family $\K$ of compact subsets of a topological space $X$ and a cometrizable space $Y$, we shall prove that the function space $C_\K(X,Y)$ is cometrizable.

The family $\K$, being separable, contains a countable subfamily $\mathcal D\subseteq\K$ such that for every $K\in\K$ there exists $\tilde K\in\K$ such that $K$ is contained in the closure of the set $\tilde K\cap\cup\mathcal D$. Then the density of $\bigcup\K$ in $X$ implies the density of the union $\bigcup\mathcal D$ in $X$.

 Since the space $Y$ is cometrizable, there exists a weaker metrizable topology $\tau$ on $Y$ such that for each open set $U\subseteq X$ and point $y\in U$ there exists an open neighborhood $V\subseteq Y$ of $y$ whose closure $\overline{V}^\tau$ in the topology $\tau$ is contained $U$. 
Denote by $Y_\tau$ the metrizable topological space $(Y,\tau)$.

By \cite[4.2.17]{Eng}, for every (compact) set $D\in\mathcal D$ the function space $C_{k}(D,Y_\tau)$ is metrizable. Since the union $\bigcup\mathcal D$ is dense in $X$, the map
$$r:C_\K(X,Y)\to\prod_{D\in\mathcal D}C_{k}(D,Y_\tau),\;\;r:f\mapsto (f{\restriction} D)_{D\in\mathcal D}$$
is injective.
Let $\sigma$ be the (metrizable) topology on $C_{k}(X,Y)$ such that the map $$r:(C_\K(X,Y),\sigma)\to \prod_{D\in\mathcal D}C_{k}(D,Y_\tau)$$ is a topological embedding. We claim that the topology $\sigma$ witnesses that the space $C_\K(X,Y)$ is cometrizable.

Fix any function $f\in C_\K(X,Y)$ and an open neighborhood $O_f\subseteq C_{\K}(X,Y)$. Without loss of generality, $O_f$ is of basic form $O_f=\bigcap_{i=1}^n[K_i,U_i]$ for some non-empty compact sets $K_1,\dots,K_n\in\K$ and some open sets $U_1,\dots,U_n\subseteq Y$. For every $i\le n$ and point $x\in K_i$, find a neighborhood $V_{f(x)}\subseteq Y$ of $f(x)\in U_i$ whose $\tau$-closure $\overline{V}^\tau_{f(x)}$ is contained in $U_i$. Using the regularity of the cometrizable space $Y$, find two open neighborhoods $N_{f(x)},W_{f(x)}$ of $f(x)$ such that $\overline{N}_{f(x)}\subseteq W_{f(x)}\subset\overline{W}_{f(x)}\subseteq V_{f(x)}$.

By the compactness of $K_i$ the open cover $\{f^{-1}(N_{f(x)}):x\in K_i\}$ of $K_i$ has a finite subcover $\{f^{-1}(N_{f(x)}):x\in F_i\}$ (here $F_i\subseteq K_i$ is a suitable finite subset of $K_i$). By choice of the family $\mathcal D$, for every $x\in F_i$, the compact set $K_{i,x}:=K_i\cap f^{-1}(\overline{N}_{f(x)})\in\K$ can be enlarged to a compact set $\tilde K_{i,x}\in \K$ such that $K_{i,x}$ is contained in the closure of the set $\tilde K_{i,x}\cap\bigcup\D$. Replacing the set $\tilde K_{i,x}$ by $\tilde K_{i,x}\cap f^{-1}(\overline{W}_{\!f(x)})$, we can assume that $f(\tilde K_{i,x})\subseteq\overline{W}_{\!f(x)}\subseteq V_{f(x)}$.

Consider the open neighborhood $$V_f=\bigcap_{i=1}^n\bigcap_{x\in F_i}[\tilde K_{i,x},V_{f(x)}]$$of $f$ in the function space $C_\K(X,Y)$. We claim that its $\sigma$-closure $\overline{V}^\sigma_f$ is contained in $O_f$. 

Given any function $g\notin O_f$, we should find a neighborhood $O_g\in\sigma$ of $g$ that does not intersect $V_f$. Since $g\notin O_f$, there exists $i\le n$ and a point $z\in K_i$ such that $g(z)\notin U_i$. Find a point $x\in F_i$ with $z\in K_{i,x}$. Taking into account that $\overline{V}^\tau_{f(x)}\subseteq U_i\subseteq Y\setminus\{g(z)\}$, we conclude that $g(z)\notin \overline{V}^\tau_{f(x)}$. Since $z\in K_{i,n}\subset\overline{\tilde K_{i,n}\cap\bigcup\D}$, the continuity of the function $g:X\to Y_\tau$ yields a point $d\in \tilde K_{i,n}\cap\bigcup \D$ such that $g(d)\notin  \overline{V}^\tau_{f(x)}$. Then $O_g:=[\{d\},Y\setminus \overline{V}^\tau_{f(x)}]\in\sigma$ is a required $\sigma$-open neighborhood of $g$ that is disjoint with the set  $V_f$.

\section{Proof of Theorem~\ref{t:main}}\label{s:main}

The three statements of Theorem~\ref{t:main} are proved in the following three lemmas.

\begin{lemma}\label{l1} Each cosmic space $X$ is $\cs$-separable.
\end{lemma}

\begin{proof} By \cite[4.9]{Grue}, the cosmic space $X$ is the image of a separable metrizable space $M$ under a continuous map $f:M\to X$. Let $D$ be any countable dense set in $M$. We claim that its image $f(D)$ is sequentially dense in $X$. Indeed, given any point $x\in X$, we can find a point $z\in M$ with $f(z)=x$ and choose a sequence $\{z_n\}_{n\in\w}\subseteq D$ that conveges to $z$. Then the sequence $\{f(z_n)\}_{n\in\w}\subseteq f(D)$ converges to $x$ in the space $X$.
\end{proof}

We recall that the class of $\cs$-cosmic spaces coincides with the class of $\aleph_0$-spaces.

\begin{lemma} Each $\cs$-cosmic space $X$ is $\mathsf k$-separable.
\end{lemma}

\begin{proof} By \cite[p.494]{Grue}, the $\aleph_0$-space $X$ is the image of a separable metric space $(M,d)$ under a compact-covering map $f:M\to X$. The compact-covering property of $f$ means that each compact set $K\subseteq X$ coincides with the image $f(C)$ of some compact set $C\subseteq M$. Let $D$ be any countable dense set in $M$. We claim that its image $f(D)$ is $k$-dense in $X$. Indeed, given any compact set $K\subseteq X$, use the compact-covering property of $f$ to find a compact set $C\subseteq M$ with $f(C)=K$. Fix a countable dense set $\{c_n\}_{n\in\w}$ in $C$ and for every $n,k\in\w$ choose a point $c_{n,k}\in D$ such that $d(c_n,c_{n,k})<\frac1{2^{n+k}}$. It is easy to see that the set $\tilde C=C\cup\{c_{n,k}:n,k\in\w\}$ is compact and $D\cap \tilde C\supset\{x_{n,k}:n,k\in\w\}$ is dense in $\tilde C$. Then $\tilde K=f(\tilde C)$ is a compact set, containing $K$ and the set $\tilde K\cap f(D)\supseteq\{f(c_{n,k})\}_{n,k\in\w}$ is dense in $\tilde K$. This shows that the countable set $f(D)$ is $k$-dense in $X$ and the space $X$ is $\mathsf k$-separable.
\end{proof}

Our last lemma proves the (most difficult) third statement of Theorem~\ref{t:main}.

\begin{lemma}\label{l4} Each $\mathsf{as}$-cosmic space $X$ is cometrizable.
\end{lemma}

\begin{proof}  Fix a countable $\mathsf{as}$-network $\mathcal N$ for the $\mathsf{as}$-cosmic space $X$. Let $\mathcal B$ be a countable base of the topology of the real line $\IR$ such that $\mathcal B$ is closed under finite unions. Let $\tau$ be the zero-dimensional Hausdorff topology on the function space $C(X,\IR)$, generated by the countable subbase consisting of the sets $$[N,B]:=\{f\in C(X,\IR):f(N)\subseteq B\}\mbox{ and }C(X,\IR)\setminus[N,B]$$
where $N\in\mathcal N$ and $B\in\mathcal B$. Denote by $C_\tau(X)$ the space $C(X,\IR)$ endowed with the topology $\tau$.

Let $D$ be any countable dense subset of the second-countable space $C_\tau(X)$. Let $\mathcal K$ be the family of compact subsets $K\subseteq C_\tau(X)$ such that either $K$ is finite or $K$ has a unique non-isolated point and $K\cap D$ is dense in $K$. 

It is easy to see that $\K$ is a separable family of compact sets in $C_\tau(X)$.
By Theorem~\ref{t:Ck}, the function space $C_\K(C_\tau(X))$ is cometrizable.

Consider the canonical map $\delta:X\to C_\K(C_\tau(X))$ assigning to each point $x\in X$ the Dirac measure $\delta_x:f\mapsto f(x)$. Let us show that the Dirac measure $\delta_x$ is a continuous function on $C_\tau(X)$. Given any function $f\in C_\tau(X)$ and an open neighborhood $U\in\mathcal B$ of $\delta_x(f)=f(x)$, find a set $N\in\mathcal N$ such that $x\in N\subseteq f^{-1}(U)$. Then $[N,U]\in\tau$ is a  neighborhood of $f$ such that $\delta_x([N,U])\subseteq U$. This means that the functional $\delta_x$ is continuous and the map $\delta:X\to C_{\K}(C_\tau(X))$ is well-defined.

Let us show that the inverse map $\delta^{-1}:\delta(X)\to X$ is continuous. Take any point $x\in X$ and fix any neighborhood $O_x\subseteq X$ of $x$. The space $X$, being regular and Lindel\"of (because of cosmic), is normal. Then we can find a continuous function $f:X\to [0,1]$ such that $f(x)=1$ and $f(X\setminus O_x)\subset\{0\}$. Since the family $\K$ contains all singletons in the space $C_\tau(X)$, the singleton $\{f\}$ belongs to the family $\K$. Consider the open set $U=\{r\in\IR:r>\frac12\}$ and the open set $[\{f\},U]\subseteq C_\K(C_\tau(X))$, which contains the functional $\delta_x$ as $\delta_x(f)=f(x)=1\in U$. Since $$\delta^{-1}([\{f\};U])=\{z\in X:\delta_z\in [\{f\};U]\}=\{z\in X:\delta_z(f)\in U\}=\{z\in X:f(x)>\tfrac12\}\subseteq O_x,$$the function $\delta^{-1}:\delta(X)\to X$ is continuous at the point $\delta_x\in C_\K(C_\tau(X))$.

It remains to prove that the map $\delta:X\to C_\K(C_\tau(X))$ is continuous. Since the base $\mathcal B$ is closed under finite unions, it suffices to prove that for every compact set $K\in\K$ and every basic open  set $U\in\mathcal B$ the preimage $\delta^{-1}([K,U])$ is open in $X$. By the definition of the family $\K$, the compact set $K$ is either finite or has a unique non-isolated point and $K\cap D$ is dense in $D$.

First assume that $K$ is finite. Then $$\delta^{-1}([K,U])=\{x\in X:\delta_x\in [K,U]\}=\bigcap_{f\in K}\{x\in X:\delta_x(f)\in U\}=\bigcap_{f\in K}f^{-1}(U)$$ is open in $U$ by the continuity of the functions $f\in K$.

Now assume that $K$ has a unique non-isolated point  and the set $K\cap D$ is dense in $K$. Then $K=\{f_\infty\}\cup\{f_n\}_{n\in\w}$ for some sequence of functions $\{f_n\}_{n\in\w}\subseteq D$ that converge to a function $f_\infty$ in the space $C_\tau(X)$. Assuming that the set $\delta^{-1}([K,U])$ is not open in $X$, we conclude that $\delta^{-1}([K,U])=\bigcap_{f\in K}f^{-1}(U)$ is not a neighborhood of some point $x\in \delta^{-1}[K,U])$. Then each neighborhood $O_x$ of $x$ intersects infinitely many closed sets $X\setminus f_n^{-1}(U)$. Since $\mathcal N$ is an $\mathsf{as}$-network, there exists a set $N\in\mathcal N$ such that $N\subseteq f_\infty^{-1}(U)$ and $N$ intersects infinitely many closed sets $X\setminus f_n^{-1}(U)$. Since $[N,U]$ is a basic open neighborhood of the function $f_\infty$ in the space $C_\tau(X)$, there exists $m\in\w$ such that $f_n\in [N,U]$ for all $n\ge m$. Then $N\subseteq f_n^{-1}(U)$ for all $n\ge m$ and $N$ cannot intersect infinitely many sets $X\setminus f_n^{-1}(U)$. This contradiction completes the proof of the continuity of the map $\delta:X\to C_\K(C_\tau(X))$.

Therefore, $\delta$ is a topological embedding of the space $X$ into the cometrizable space $C_\K(C_\tau(X))$, which implies that the space $X$ is cometrizable.
\end{proof}

 \section{The construction of the space from Example~\ref{ex}}\label{s:ex}

In this section we construct a regular topology $\tau$ of weight $\w_1$ on the set of rational numbers $\IQ$ such that the topological space $(\IQ,\tau)$ is not cometrizable (and contains no infinite compact subsets under CH). 

Let $\tau_0$ be the standard metrizable topology on the set $\IQ$ of rational numbers. Let $\mathfrak S$ be the set $\mathfrak S$ of all injective functions $s:\w\to\IQ\setminus\{0\}$ such that the sequence $(s(n))_{n\in\w}$ converges in the topological space $(\IQ,\tau_0)$.

Since the set $\mathfrak S$ has cardinality $|\mathfrak S|=\mathfrak c$, it can be written as $\mathfrak S=\{s_{\alpha+1}\}_{\alpha\in\mathfrak c}$. For each limit ordinal $\alpha<\w_1$ let $s_\alpha:\w\to\IQ$ be the map defined by $s_\alpha(n)=n+1$ for $n\in\w$.

 By transfinite induction of length $\w_1$, for every countable ordinal $\alpha$ we select a regular second countable topology $\tau_\alpha$ on $\IQ$ and a neighborhood $U_\alpha\in\tau_\alpha$ of zero such that the following conditions are satisfied for every $\alpha<\w_1$:
\begin{itemize}
\item[$(1_\alpha)$] $\bigcup_{\beta<\alpha}\tau_\beta\subset\tau_\alpha$;
\item[$(2_\alpha)$] the topological space $(\IQ,\tau_\alpha)$ has no isolated points;
\item[$(3_\alpha)$] for every $\beta<\alpha$ and every neighborhood $V\in\tau_\alpha$ of zero the $\tau_\beta$-closure of $V$ is not contained in $U_{\beta+1}$;
\item[$(4_\alpha)$] the sequence $(s_\alpha(n))_{n\in\w}$ is not convergent in the topological space $(\IQ,\tau_\alpha)$.
\end{itemize}
Assume that for some non-zero ordinal $\alpha$ and all ordinals $\beta<\alpha$ we have constructed topologies $\tau_\beta$ and open sets $U_\beta\in\tau_\beta$ satisfying the conditions $(1_\beta)$--$(4_\beta)$. 

If $\alpha$ is a limit ordinal, let $\tau_\alpha$ be the topology on $\IQ$, generated by the base $\bigcup_{\beta<\alpha}\tau_\beta$, and observe that the conditions $(1_\alpha)$--$(4_\alpha)$ are satisfied. The condition $(4_\alpha)$ is satisfied since $s_\alpha(n)=n+1$ for $n\in\w$.

Now assume that $\alpha$ is a successor ordinal and hence $\alpha=\gamma+1$ for some ordinal $\gamma$. Let $\xi:\w\to\gamma=[0,\gamma)$ be a function such that for every $\beta<\gamma$ the preimage $\xi^{-1}(\beta)=\{n\in\w:\xi(n)=\beta\}$ is infinite.

Let $\{V_n\}_{n\in\w}\subset\tau_\gamma$ be a neighborhood base at zero such that $V_0=\IQ$ and $$V_{n+1}\subseteq V_{n}\cap[2^{-n-1},2^{n+1}]\cap U_{\xi(n)+1}$$ for all $n\in\w$. Since the metrizable countable space $(\IQ,\tau_\gamma)$ is zero-dimensional, we can additionally assume that each $V_n$ is closed-and-open in the topology $\tau_\gamma$. If the limit point $x_0$ of the convergent sequence $(s_{\gamma+1}(n))_{n\in\w}$ is not equal zero, then we shall additionally assume that the $\tau_0$-closure of the set $V_1$ is disjoint with the compact set $\{x_0\}\cup\{s_\alpha(n)\}_{n\in\w}$ (which does not contain zero).

If  $x_0\ne0$, then put $x_{0,k}=s_\alpha(k)$ for all $k\in\w$ and observe that the sequence $(x_{0,k})_{k\in\w}$ converges to the point $x_0$ in the topology $\tau_0$.

If $x_0=0$ and the sequence $(s_\alpha(n))_{n\in\w}$ converges to zero in the topology $\tau_\gamma$, then we can choose an increasing number sequence $(n_k)_{k\in\w}$ such that the point $x_{0,k}:=s_\alpha(n_k)$ belongs to the neighborhood $V_k\in\tau_\gamma$ of zero. In this case the sequence $(x_{0,k})_{k\in\w}$ converges to zero in the topology $\tau_\gamma$.

If $x_0=0$ but the sequence $(s_\alpha(n))_{n\in\w}$ does not converges to zero in the topology $\tau_\gamma$, then for every $k\in\w$ choose any point $x_{0,k}\in V_k$ and observe that the sequence $(x_{0,k})_{k\in\w}$ converges to zero in the topology $\tau_\gamma$. 

By the condition $(3_\gamma)$, for every $n\in\IN$ the $\tau_{\xi(n)}$-closure of $V_n$ is not contained in $U_{\xi(n)+1}$. So, we can find a sequence $\{x_{n,k}\}_{k\in\w}\subseteq V_n$ of pairwise distinct points that converge to some point $x_{n}\in X\setminus U_{\xi(n)+1}$ in the topology $\tau_{\xi(n)}$. Since $\tau_0\subset\tau_{\xi(n)+1}$, the sequence $(x_{n,k})_{k\in\w}$ converges to $x_n$ in the Euclidean topology $\tau_0$. Since $V_n\subset[2^{-n},2^{n}]$, the point $x_n$ belongs to the closed interval $[2^{-n},2^n]$ and hence the sequence $(x_n)_{n\in\w}$ converges to zero in the topology $\tau_0$. Replacing each sequence $(x_{n,k})_{k\in\w}$ by a suitable subsequence, we can assume that the points $x_{n,k}$, $(n,k)\in\IN\times\w$, are pairwise distinct and do not belong to the compact set $\{x_0\}\cup\{x_{0,k}\}_{k\in\w}$. Since $\{x_n\}\cup\{x_{n,k}\}_{n,k\in\w}\subseteq [2^{-n},2^n]$ for all $n\in\IN$, the subspace $\{x_{n,k}\}_{n,k\in\w}$ is closed and discrete in the (zero-dimensional) subspace $\IQ\setminus(\{0\}\cup\{x_n\}_{n\in\w})$ of $(\IQ,\tau_0)$.
Consequently, for each $n,k\in\w$ we can find a closed-and-open neighborhood $O_{n,k}\in\tau_0$ of the point $x_{n,k}$ such that the sets $O_{n,k}$, $n,k\in\w$, are pairwise disjoint and also are disjoint with the compact set $\{0\}\cup\{x_n\}_{n\in\w}$. Since the space $(\IQ,\tau_\gamma)$ has no isolated points, for every $n\in\IN$ and $k\in\w$ we can choose an open neighborhood $V_{n,k}\in\tau_\gamma$ of the point $x_{n,k}$ such that $V_{n,k}\subseteq O_{n,k}\cap V_n$, $V_{n,k}$ is not closed in the topology $\tau_\gamma$ and the space $O_{n,k}\setminus V_{n,k}$ has no isolated points in the topology $\tau_\gamma$.

Let $$U_\alpha:=\{0\}\cup\bigcup_{k\in\w}O_{0,2k}\cup\bigcup_{n\in\IN}\bigcup_{k\in\w}V_{n,k}$$ and let $\tau_\alpha$ be the topology generated by the subbase $\tau_\gamma\cup \{U_\alpha,\IQ\setminus U_\alpha\}$.
It is easy to check that for the topology $\tau_\alpha$ and the neighborhood $U_\alpha\in\tau_\alpha$ of zero the conditions $(1_\alpha)$--$(3_\alpha)$ are satisfied.

Let us show that the condition $(4_\alpha)$ is satisfied, too. To derive a contradiction, assume that the sequence $(s_\alpha(k))_{k\in\w}$ converges in the topological space $(X,\tau_\alpha)$. Since $\tau_\gamma\subseteq\tau_\alpha$, it converges in the space $(X,\tau_\gamma)$ and hence $(x_{0,k})_{k\in\w}$ is a subsequence of the sequence $(s_\alpha(n))_{n\in\w}$ by the definition of the points $x_{0,k}$, $k\in\w$. Since the closed-and-open set $U_\alpha\in\tau_\alpha$ contains the points $x_{0,2k}$, $k\in\w$, and does not contains the points $x_{0,2k+1}$, $k\in\w$, the sequence $(x_{0,k})_{k\in\w}$ is not convergent in the topology $\tau_\alpha$ and then the sequence $\{s_\alpha(n)\}_{n\in\w}\supset\{x_{0,k}\}_{k\in\w}$ also cannot be convergent in the topology $\tau_\alpha$. Therefore, $(4_\alpha)$ is satisfied.
\smallskip

After completing the inductive construction, consider the topology $\tau=\bigcup_{\alpha\in\w_1}\tau_\alpha$ on $\IQ$ (this is a topology as $\IQ$ is countable). We claim that the space $(\IQ,\tau)$ is not cometrizable. In the opposite case we could find a metrizable topology $\sigma\subseteq\tau$ such that for every neighborhood $U\in\tau$ of zero there exists a neighborhood $V\in\tau$ of zero whose closure in the topology $\sigma$ is contained in $U$. Since the topology $\sigma\subseteq\tau$ has a countable base, there exists a countable ordinal $\beta$ such that $\sigma\subseteq\tau_\beta$. By the choice of the topology $\sigma$, for the neighborhood $U_{\beta+1}\in\tau_{\beta+1}\subseteq\tau$, there exists a neighborhood $V\in\tau$ whose $\sigma$-closure in contained in $U_{\beta+1}$. Since $V\in\tau=\bigcup_{\alpha\in\w_1}\tau_\alpha$, there exists a countable ordinal $\alpha>\beta$ with $V\in\tau_\alpha$. The inductive condition $(3_\alpha)$ ensures that the $\tau_\beta$-closure of $V$ is not contained in $U_{\beta+1}$. Since $\sigma\subseteq\tau_\beta$, the $\sigma$-closure of $V$ contains the $\tau_\beta$-closure of $V$ and hence also is not contained in $U_{\beta+1}$. This contradiction shows that the space $(\IQ,\tau)$ is not cometrizable.
\smallskip

Finally, assuming that $\w_1=\mathfrak c$, we shall prove that the space $(\IQ,\tau)$ contains no infinite compact subsets. In the opposite case, we could find an injective function $s:\w\to \IQ\setminus\{0\}$ such that the sequence $(s(n))_{n\in\w}$ is convergent in the topological space $(\IQ,\tau)$. Since $\tau_0\subseteq\tau$, this sequence remains convergent in the space $(\IQ,\tau_0)$ and hence $s\in\mathfrak S=\{s_{\alpha+1}\}_{\alpha\in\w_1}$. So, we can find a successor ordinal $\alpha<\w_1$ such that $s=s_\alpha$. In this case the inductive condition $(4_\alpha)$ ensures that the sequence $(s_\alpha(n))_{n\in\w}$ is not convergent in the topology $\tau_\alpha$ and then the sequence $(s(n))_{n\in\w}=(s_\alpha(n))_{n\in\w}$ cannot be convergent in the topology $\tau\supset \tau_\alpha$, which contradicts the choice of $s$. Since each compact subset of the space $X$ is finite, the countable family $\mathcal N=\{\{x\}:x\in X\}$ is a $k$-network for $X$, which means that $X$ is an $\aleph_0$-space and hence is $\cs$-cosmic.

If $\w_1<\mathfrak p$, then the space $X$ is Fr\'echet-Urysohn by the definition of the cardinal $\mathfrak p$. Assuming that $X$ is an $\aleph_0$-space and applying \cite[11.4]{Grue}, we would conclude that that $X$ is stratifiable and, by Theorem~\ref{t:S}, $X$ is cometrizable, which is not the case.




\section{Acknowledgement}

The authors express their sincere thanks to Alex Ravsky for careful reading the manuscript and many valuable discussions (especially on the construction of Example~\ref{ex}), to Jerzy K\c akol whose inspiring questions resulted in the proof of Theorem~\ref{t:Ck}, and to Paul Gartside for the information on his joint paper \cite{GR} with Reznichenko.


\end{document}